\documentclass[12pt]{amsart}
\usepackage{hyperref} 
\usepackage{amsmath, amsthm, amssymb}
\usepackage{hyperref} 
\usepackage{enumerate}
\usepackage{verbatim}
\usepackage{esint}
\usepackage[T1]{fontenc}

\usepackage{tikz,amsthm,amsmath,amstext,amssymb,amscd,epsfig,euscript, mathrsfs,
 dsfont,pspicture,
multicol,graphpap,graphics,graphicx,times,enumerate,subfig,
sidecap,
wrapfig,color,pict2e}
\usepackage{setspace}

\numberwithin{equation}{section}

\title{Lower bounds on mapping content and quantitative factorization through trees}
\date{\today}
\author{Guy C. David}
\address{Department of Mathematical Sciences\\ Ball State University, Muncie, IN 47306}
\email{gcdavid@bsu.edu}

\author{Raanan Schul}
\address{Department of Mathematics\\ Stony Brook University\\ Stony Brook, NY 11794-3651}
\email{schul@math.sunysb.edu}

\thanks{G.~ C.~ David was partially supported by the National Science Foundation under Grants No. DMS-1758709 and DMS-2054004. R.~ Schul was partially supported by the National Science Foundation under Grant No. DMS-1763973.}

\begin{document}
\maketitle

\theoremstyle{plain}
\newtheorem{theorem}{Theorem}
\newtheorem{exercise}{Exercise}
\newtheorem{corollary}[theorem]{Corollary}
\newtheorem{scholium}[theorem]{Scholium}
\newtheorem{claim}[theorem]{Claim}
\newtheorem{lemma}[theorem]{Lemma}
\newtheorem{sublemma}[theorem]{Lemma}
\newtheorem{proposition}[theorem]{Proposition}
\newtheorem{conjecture}[theorem]{Conjecture}
\newtheorem{maintheorem}{Theorem}
\newtheorem{maincor}[maintheorem]{Corollary}
\newtheorem{mainproposition}[maintheorem]{Proposition}
\renewcommand{\themaintheorem}{\Alph{maintheorem}}

\theoremstyle{definition}
\newtheorem{fact}[theorem]{Fact}
\newtheorem{example}[theorem]{Example}
\newtheorem{definition}[theorem]{Definition}
\newtheorem{remark}[theorem]{Remark}
\newtheorem{question}[theorem]{Question}

\numberwithin{equation}{section}
\numberwithin{theorem}{section}

\newcommand{\cG}{\mathcal{G}}
\newcommand{\RR}{\mathbb{R}}
\newcommand{\HH}{\mathcal{H}}
\newcommand{\LIP}{\textnormal{LIP}}
\newcommand{\Lip}{\textnormal{Lip}}
\newcommand{\Tan}{\textnormal{Tan}}
\newcommand{\length}{\textnormal{length}}
\newcommand{\dist}{\textnormal{dist}}
\newcommand{\diam}{\textnormal{diam}}
\newcommand{\vol}{\textnormal{vol}}
\newcommand{\rad}{\textnormal{rad}}
\newcommand{\side}{\textnormal{side}}

\def\bA{{\mathbb{A}}}
\def\bB{{\mathbb{B}}}
\def\bC{{\mathbb{C}}}
\def\bD{{\mathbb{D}}}
\def\bR{{\mathbb{R}}}
\def\bS{{\mathbb{S}}}
\def\bO{{\mathbb{O}}}
\def\bE{{\mathbb{E}}}
\def\bF{{\mathbb{F}}}
\def\bH{{\mathbb{H}}}
\def\bI{{\mathbb{I}}}
\def\bT{{\mathbb{T}}}
\def\bZ{{\mathbb{Z}}}
\def\bX{{\mathbb{X}}}
\def\bP{{\mathbb{P}}}
\def\bN{{\mathbb{N}}}
\def\bQ{{\mathbb{Q}}}
\def\bK{{\mathbb{K}}}
\def\bG{{\mathbb{G}}}

\def\nrj{{\mathcal{E}}}
\def\cA{{\mathscr{A}}}
\def\cB{{\mathscr{B}}}
\def\cC{{\mathscr{C}}}
\def\cD{{\mathscr{D}}}
\def\cE{{\mathscr{E}}}
\def\cF{{\mathscr{F}}}
\def\cB{{\mathscr{G}}}
\def\cH{{\mathcal{H}}}
\def\cI{{\mathscr{I}}}
\def\cJ{{\mathscr{J}}}
\def\cK{{\mathscr{K}}}
\def\Layer{{\rm Layer}}
\def\cM{{\mathscr{M}}}
\def\cN{{\mathscr{N}}}
\def\cO{{\mathscr{O}}}
\def\cP{{\mathscr{P}}}
\def\cQ{{\mathscr{Q}}}
\def\cR{{\mathscr{R}}}
\def\cS{{\mathscr{S}}}
\def\Up{{\rm Up}}
\def\cU{{\mathscr{U}}}
\def\cV{{\mathscr{V}}}
\def\cW{{\mathscr{W}}}
\def\cX{{\mathscr{X}}}
\def\cY{{\mathscr{Y}}}
\def\cZ{{\mathscr{Z}}}

  \def\del{\partial}
  \def\diam{{\rm diam}}
	\def\VV{{\mathcal{V}}}
	\def\FF{{\mathcal{F}}}
	\def\QQ{{\mathcal{Q}}}
	\def\BB{{\mathcal{B}}}
	\def\XX{{\mathcal{X}}}
	\def\PP{{\mathcal{P}}}

  \def\del{\partial}
  \def\diam{{\rm diam}}
	\def\image{{\rm Image}}
	\def\domain{{\rm Domain}}
  \def\dist{{\rm dist}}
	\newcommand{\Gr}{\mathbf{Gr}}
\newcommand{\md}{\textnormal{md}}
\newcommand{\vspan}{\textnormal{span}}

\newcommand{\RS}[1]{{  \color{blue} \textbf{Raanan:} #1}}
\newcommand{\GCD}[1]{{  \color{red} \textbf{Guy:} #1}}

\begin{abstract}
We give a simple quantitative condition, involving the ``mapping content'' of Azzam--Schul, that implies that a Lipschitz map from a Euclidean space to a metric space must be close to factoring through a tree. Using results of Azzam--Schul and the present authors, this gives simple checkable conditions for a Lipschitz map to have a large piece of its domain on which it behaves like an orthogonal projection. The proof involves new lower bounds and continuity statements for mapping content, and relies on a ``qualitative'' version of the main theorem recently proven by Esmayli--Haj\l asz.
\end{abstract}

\maketitle

\section{Introduction}

The goal of this paper is to explain a link between two questions that one can ask about a Lipschitz mapping 
from a Euclidean space into a metric space. We first describe these questions philosophically, and then give precise details.

The first question is: Are there suitable coordinates on which $f$ looks like a simple orthogonal projection? 
From a geometric perspective, the nice properties of such a projection are that the fibers of the mapping are parallel and are parametrized in a uniform way by Euclidean spaces of the expected dimension. 

As one cannot expect every Lipschitz map to look like a projection in coordinates, we ask instead:
\begin{question}\label{q:vague1}
Are there \textit{large pieces} of the domain on which the restriction of $f$ admits such a change of coordinates?
\end{question}
In other words, we are looking for a large piece of the domain on which we have a (non-smooth) implicit function theorem. This question was explored in detail recently in \cite{AS12, HZ, David-Schul:harder-sarder}, as we make precise and explain further below.

It turns out that there are highly non-trivial mappings -- first discovered by Kaufman \cite{Ka} -- that fail to admit any such pieces at all. Kaufman's example is a Lipschitz map of $[0,1]^3$ onto $[0,1]^2$ that nonetheless fails in a precise way to behave like an orthogonal projection $\RR^3\rightarrow\RR^2$ on \textit{any} set of positive measure.

Kaufman's example provides the link to our second question, because it \textit{factors through a tree}.

\begin{definition}
A \textit{metric tree} is a geodesic metric space in which every geodesic triangle is isometric to a tripod.

We say that a Lipschitz map $f\colon X \rightarrow Y$ \textit{factors through a tree} if there is a metric tree $T$ and two Lipschitz maps $g\colon X \rightarrow T$ and $h\colon T \rightarrow Y$ such that $f=h \circ g$.

If the two maps $h$ and $g$ are both $L$-Lipschitz, we say that $f$ \textit{factors through a tree by $L$-Lipschitz maps}.
\end{definition}

Thus, the second question we ask about a given Lipschitz map $f$ is:

\begin{question}\label{q:vague2}
How can we tell whether $f$ factors through a tree, or (taking the quantitative viewpoint) is close to factoring through a tree? 
\end{question}

Our initial motivation for this question was to understand whether Kaufman-type examples are the only ones that can provide negative answers to Question \ref{q:vague1} in low-dimensional situations.

Maps that factor through trees are degenerate in some topological ways. As a simple illustration, the restriction of such a map to any Jordan curve in the domain must fail to be injective. Nonetheless, such mappings can exhibit unexpected geometric complexity, as shown by Kaufman's example, whose behavior is quite surprising in view of the one-dimensionality of trees. Further examples of mappings that factor through trees and nonetheless have ``large'' images appear in \cite[Section 2]{AS12}.

It turns out that the answers to both Questions \ref{q:vague1} and \ref{q:vague2} are governed by the same quantity, the so-called ``$(n,m)$-mapping content'' first introduced in \cite{AS12}. 

\begin{definition}\label{def:mappingcontent}
Let $n,m$ be non-negative integers and let $E\subseteq Q_0 = [0,1]^{n+m}$ be a set. We define the ``$(n,m)$-mapping content'' of $f$ on $E$  as
$$\cH^{n,m}_\infty(f,E):= \inf \sum_{Q_i} \cH^n_\infty(f(Q_i))\side(Q_i)^m,$$
where the infimum is taken over all coverings $\{Q_i\}$ of $E$ by dyadic cubes with disjoint interiors in $Q_0$.
\end{definition}
Here $\cH^n_\infty$ refers to the $n$-dimensional Hausdorff content, defined in subsection \ref{subsec:content}.
\begin{remark}
We note that this definition of $\cH^{n,m}_\infty$ is precisely that used in \cite{David-Schul:harder-sarder}. It differs slightly from that stated in \cite[Equation (1.3)]{AS12} in that the infimum is over coverings by dyadic cubes rather than arbitrary cubes, although the results in \cite{AS12} really only require dyadic cubes. The two definitions are weakly comparable, in the sense that if one quantity is small than the other is small; this follows from \cite[Corollary E]{David-Schul:harder-sarder}.
\end{remark}

The quantity $\cH^{n,m}_\infty$ serves in some sense as a ``coarse'' (and metric) analog of the $L^1$-norm of the $n$-dimensional Jacobian of $f$, defined as the product of the $n$ largest singular values of the derivative of $f$. Thus, a mapping with small $\HH^{n,m}_\infty$ might be viewed as highly compressing $n$-dimensional volumes on some scales. This must be interpreted with care, however, given Kaufman's example, which has $\HH^{2,1}_\infty = 0$ but maps onto the unit square $[0,1]^2$. This is discussed further in \cite[Section 3]{David-Schul:harder-sarder}. 

When $n=2$, the present authors made the following conjecture in \cite[Conjecture 1.13]{David-Schul:harder-sarder}:
\begin{conjecture}[\cite{David-Schul:harder-sarder}]\label{conj:21}
Let $Q_0=[0,1]^{2+m}$ and let $f\colon Q_0 \rightarrow \ell^\infty$ be a $1$-Lipschitz mapping.
\begin{itemize}
\item (Qualitative version) If $\HH^{2,m}_\infty(f,Q_0)=0$, then $f$ factors through a tree.

\item (Quantitative version) For every $\epsilon>0$, there is a $\delta=\delta(\epsilon,m)$ with the following property: If $\HH^{2,m}_\infty(f,Q_0)<\delta$, then there is a $1$-Lipschitz map $g\colon Q_0 \rightarrow \ell^\infty$ that factors through a tree and satisfies $\sup_{Q_0}\|g-f\|_\infty<\epsilon$.
\end{itemize}

\end{conjecture}

This conjecture was inspired by the techniques of \cite{WY} and the evidence of Kaufman's example, which provides a negative answer to Question \ref{q:vague1}, factors through a tree, and has $\HH^{2,1}_\infty$ equal to zero. The same properties hold for the examples in \cite[Section 2]{AS12}. Essentially, the conjecture asks whether Kaufman-type examples are the only examples of mappings with small mapping content, and hence mappings providing negative answers to Question \ref{q:vague1}, in the case $n=2$. 

\begin{remark}\label{rem:othern}
In the case $n\leq 1$ or $m=0$, mappings from $[0,1]^{n+m}$ with vanishing or small $\HH^{n,m}_\infty$ are easy to characterize, while in the case $n>2$ the problem appears difficult due to the constructions of topologically non-trivial low-rank Lipschitz mappings in \cite{WY, GHP}. Thus, the case $n=2$ that we address here is simply the first remaining open problem in a long list.  This is discussed further in \cite[Section 1.4]{David-Schul:harder-sarder}.
\end{remark}

Recently, Esmayli--Haj\l asz \cite{EGH-personal} showed that a mapping with $\HH^{2,m}_\infty$ equal to zero \textbf{must} factor through a tree, answering the qualitative cases of Question \ref{q:vague2} and, equivalently, Conjecture \ref{conj:21}. (This is the equivalence ``(e) $\Leftrightarrow$ (a)'' in \cite[Theorem 1.1]{EGH-personal}, combined with their Remark 1.2.)

\begin{theorem}[Esmayli--Haj\l asz \cite{EGH-personal}]\label{thm:EGH}
Let $Q_0=[0,1]^{2+m}$ and $f\colon Q_0 \rightarrow Y$ be a Lipschitz map into a metric space. Then $\HH^{2,m}_\infty(f,Q_0)=0$ if and only if $f$ factors through a tree. Moreover, in this case, if $f$ is $1$-Lipschitz then $f$ factors through a tree by $1$-Lipschitz maps.
\end{theorem}

Motivated by the more quantitative concerns of \cite{AS12, David-Schul:harder-sarder}, we use Theorem \ref{thm:EGH} to prove a quantitative stability theorem: maps with small, but not necessarily vanishing, mapping content must be close to factoring through trees.

Our main theorem in this paper is the following, completely resolving Question \ref{q:vague2} and finishing the proof of Conjecture \ref{conj:21}.

\begin{maintheorem}\label{thm:main}
Let $m$ be a non-negative integer. For every $\epsilon>0$, there is a $\delta=\delta(\epsilon,m)>0$ with the following property:
Let $f\colon Q_0:=[0,1]^{2+m}\rightarrow \ell^\infty$ be a $1$-Lipschitz map.
If $\HH^{2,m}_\infty(f,Q_0)<\delta$, then there is a $1$-Lipschitz map $g\colon Q_0 \rightarrow \ell^\infty$ such that 
$$\sup_{x\in Q_0} \|g(x)-f(x)\|_\infty<\epsilon$$ 
and $g$ factors through a tree by $1$-Lipschitz maps. 

Conversely, for every $\epsilon'>0$, there is a $\delta'=\delta(\epsilon',m)>0$ such that if  $f,g\colon Q_0:=[0,1]^{2+m}\rightarrow \ell^\infty$ are $1$-Lipschitz maps, $g$ factors through a tree by $1$-Lipschitz maps, and
$$\sup_{x\in Q_0} \|g(x)-f(x)\|_\infty<\delta',$$
then $\HH^{2,m}_\infty(f,Q_0)<\epsilon'$. 
\end{maintheorem}

\begin{remark}
As every separable metric space embeds isometrically in $\ell^\infty$, the restriction in Conjecture \ref{conj:21} and Theorem \ref{thm:main} that the target space is $\ell^\infty$ is no restriction at all; it simply allows for a cleaner statement.
\end{remark}

Without the restriction to $n=2$ in Theorem \ref{thm:main}, the result is false, as alluded to in Remark \ref{rem:othern}. On the other hand, the assumption $n=2$ enters this paper only through Theorem \ref{thm:EGH}, so our methods in fact show the following in all dimensions.

\begin{maintheorem}\label{thm:alldim}
For every $n,m$ and $\epsilon>0$, there is a $\delta=\delta(\epsilon,n,m)>0$ with the following property: If $f\colon Q_0:=[0,1]^{n+m}\rightarrow \ell^\infty$ is a $1$-Lipschitz map and $\HH^{n,m}_\infty(f,Q_0)<\delta$, then there is a $1$-Lipschitz map $g\colon Q_0 \rightarrow \ell^\infty$ such that
$$\sup_{x\in Q_0} \|g(x)-f(x)\|_\infty<\epsilon$$
and $\HH^{n,m}_\infty(g, Q_0) = 0$.

Conversely, for every $\epsilon'>0$, there is a $\delta'=\delta(\epsilon',n,m)>0$ such that if $f,g\colon Q_0:=[0,1]^{n+m}\rightarrow \ell^\infty$ are $1$-Lipschitz maps, $\HH^{n,m}_\infty(g,Q_0)=0$, and
$$\sup_{x\in Q_0} \|g(x)-f(x)\|_\infty<\delta',$$
then $\HH^{n,m}_\infty(f,Q_0)<\epsilon'$. 
\end{maintheorem}

Thus, attempts at higher-dimensional generalizations of Theorem \ref{thm:main} could perhaps first proceed by attempting to characterize maps with vanishing mapping content in higher dimensions. Progress on this appears in \cite[Corollary 7.22]{esmayli2020coarea}.

Before proceeding with more precise motivation and a description of the remaining results of the paper, we explicitly state a simple corollary of Theorem \ref{thm:main} which does not require the notion of mapping content. This is just the case $m=0$, in which case $\HH^{2,m}_\infty(f,Q_0)$ is simply $\HH^2_\infty(f(Q_0))$, the standard $2$-dimensional Hausdorff content of the image of $f$. Thus:

\begin{corollary}\label{cor:mzero}
For every $\epsilon>0$, there is a $\delta=\delta(\epsilon)>0$ with the following property:
Let $f\colon Q_0:=[0,1]^{2}\rightarrow \ell^\infty$ be a $1$-Lipschitz map.
If $\HH^{2}_\infty(f(Q_0))<\delta$, then there is a map $g\colon Q_0 \rightarrow \ell^\infty$ such that 
$$\sup_{x\in Q_0} \|g(x)-f(x)\|_\infty<\epsilon$$ 
and $g$ factors through a tree.
\end{corollary}
In other words, a $1$-Lipschitz map from a square can only compress $2$-dimensional content insofar as it is close to factoring through a tree.

\subsection{Motivation from geometric measure theory: ``Hard Sard sets''}
As motivation, we now explain more precisely the connection between our results and Question \ref{q:vague1}.

A classical result in geometric measure theory says that for a Lipschitz map $f:\RR^{n+m}\to \RR^n$, the pre-image of $\cH^n$--almost every point of $f(\RR^{n+m})$ is $m$-rectifiable. (See, e.g., \cite[Theorem 3.2.22]{Fe}.) That is, almost every fiber of $f$ can be covered by Lipschitz images of $m$-dimensional Euclidean space, up to $m$-dimensional measure zero.

As with many problems in geometric measure theory, there is a natural quantitative follow-up question:
{\it is it possible to parametrize these fibers (or at least a \textbf{large fraction} of them) in a \textbf{uniform} way?}
More specifically,
{\it is it possible to {\bf rearrange} the domain (or at least a \textbf{large fraction} of it) so that the fibers are parallel $m$-planes}? 

The canonical example of a map whose fibers are parallel $m$-planes is a simple orthogonal projection, which connects this to Question \ref{q:vague1}. The recent works \cite{AS12, HZ, David-Schul:harder-sarder} all explore this question. We first make more precise, with a definition from \cite{AS12, David-Schul:harder-sarder}, what we mean by ``straightening out'' or ``making parallel'' the fibers of a mapping. 

\begin{definition}\label{def:HSpair}
Let $Q_0=[0,1]^{n+m}$ and $f\colon Q_0 \rightarrow Y$ be a Lipschitz map into a metric space. Let $E\subseteq Q_0=[0,1]^{n+m}$ and $g:E\rightarrow \RR^{n+m}$ be a bi-Lipschitz mapping. We call $(E,g)$ a \textbf{Hard Sard pair for $f$} if there is a constant $C_{Lip}$ such that the following conditions hold.

Write  points of $\RR^{n+m}$ as $(x,y)$ with $x\in \RR^n$ and $y\in \RR^m$. Let $F = f \circ g^{-1}$.
We ask that:
\begin{enumerate}[(i)]
\item\label{HS2} $g$ extends to a globally $C_{Lip}$-bi-Lipschitz homeomorphism from $\RR^{n+m}$ to $\RR^{n+m}$.
\item\label{HS3} If $(x,y)$ and $(x',y')$ are in $g(E)$, then $F(x,y) = F(x',y')$ if and only if $x=x'$. 
\item\label{HS4} The map
$(x,y) \mapsto (F(x,y),y)$
is $C_{Lip}$-bi-Lipschitz on the set $g(E)$. In particular, for all $y\in \RR^m$, the restriction
$F|_{(\RR^n\times\{y\}) \cap g(E)}$
is $C_{Lip}$-bi-Lipschitz.
\end{enumerate}
If $E\subseteq Q_0$ is a set and there exists a mapping $g\colon \RR^{n+m}\rightarrow \RR^{n+m}$ satisfying \eqref{HS2}-\eqref{HS4} for $E$, then we call $E$ a \textbf{Hard Sard set for $f$}.
\end{definition}
We think of $g$ as a globally bi-Lipshitz change of coordinates that ``straightens out'' the fibers of $f|_E$. 
Since the linear projection mapping $\pi(x,y)=x$ on $Q_0$ satisfies all the properties requested for the map $F=f\circ g^{-1}$ on $g(E)$ in Definition \ref{def:HSpair}, we interpret Definition \ref{def:HSpair} as saying that $f$ ``looks like a projection'' when restricted to $E$, up to globally a bi-Lipschitz change of coordinates $g$. See Figure \ref{f:figure-1}. 

\begin{figure}[t]
\includegraphics[width=0.8\textwidth]{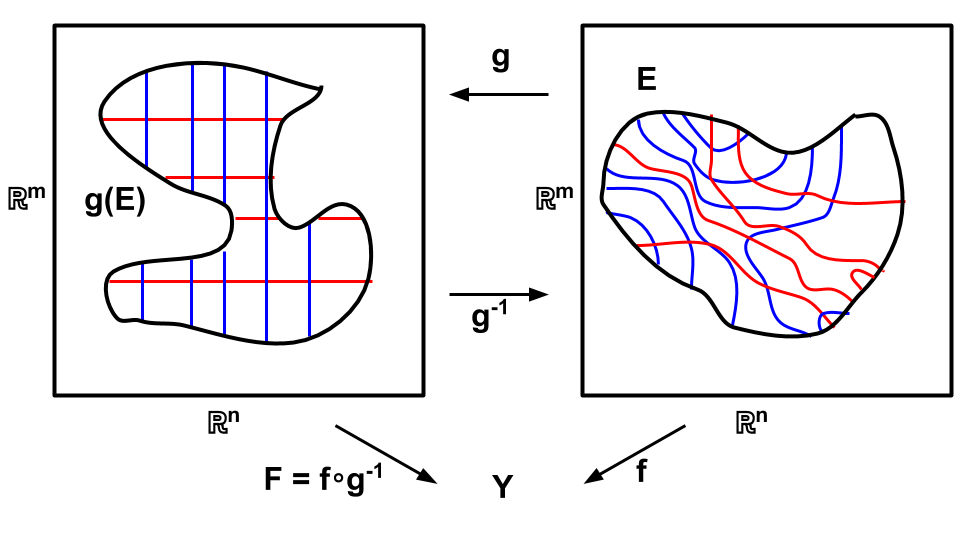}
\centering
\caption{A schematic diagram of a Hard Sard set in the case $n=m=1$. The map $g$ straightens out the (blue) fibers of $f$ and the (red) transverse curves on which $f$ is bi-Lipschitz.}\label{f:figure-1}
\end{figure}


The motivation for both of the above definitions was Theorem I of Azzam and the second author in \cite{AS12}. That result was improved in \cite{David-Schul:harder-sarder} to the following theorem.
\begin{theorem}[Theorem A of \cite{David-Schul:harder-sarder}]\label{thm:hardersarder}
Let $Q_0$ be the unit cube in $\RR^{n+m}$ and let $f\colon Q_0\rightarrow X$ be a $1$-Lipschitz map into a metric space $X$ with
$ \cH^n(f(Q_0)) \leq 1.$
Given any $\gamma>0$, we can write
$ Q_0 = E_1 \cup \dots \cup E_M \cup G,$
where $E_i$ are Hard Sard sets for $f$ and
$ \cH^{n,m}_\infty(f,G) < \gamma.$
The constant $M$ and the constants $C_{Lip}$ associated to the Hard Sard pairs $(E_i,g_i)$ depend only on $n$, $m$, and $\gamma$.
\end{theorem}
This theorem and \cite[Theorem I]{AS12} can be viewed as ``quantitative implicit function theorems'' for Lipschitz maps into metric spaces. Compare also with the qualitative theorem of \cite{HZ}.

Thus, in combination with Theorem \ref{thm:hardersarder} (or Theorem I of \cite{AS12}), our Theorem \ref{thm:main} immediately implies the following.
\begin{corollary}
Let $Q_0$ be the unit cube in $\RR^{2+m}$ and let $f\colon Q_0\rightarrow \ell^\infty$ be a $1$-Lipschitz map with $ \cH^n(f(Q_0)) \leq 1.$

If 
$$\sup_{x\in Q_0} \|g(x)-f(x)\|_\infty \geq \epsilon$$ 
for all mappings $g\colon Q_0 \rightarrow \ell^\infty$ that factor through trees, then $f$ admits a Hard Sard set $E$ with $\HH^{n+m}(E) > \gamma$, where $\gamma>0$ depends only on $m$ and $\epsilon$.
\end{corollary}
\begin{proof}
By Theorem \ref{thm:main}, we have $\HH^{n,m}_\infty(f,Q_0)$ bounded below depending only on $\epsilon$. By Theorem \ref{thm:hardersarder}, $f$ then admits a Hard Sard set $E$ with $\HH^{n,m}_\infty(f,E)$ (and the associated constants) bounded below. The $\HH^{n+m}$-measure of $E$ is then comparable to $\HH^{n,m}_\infty(f,E)$ by \cite[Lemma 3.5]{David-Schul:harder-sarder}.
\end{proof}






\subsection{Additional results on mapping content: lower bounds and continuity}

%
%

Theorem \ref{thm:main} will follow from Theorem \ref{thm:EGH} and a new lower bound (Theorem \ref{thm:lowerbound}) and continuity statement (Theorem \ref{thm:contentzero}) for mapping content, which apply for general $n,m\in\mathbb{N}$. For the remainder of the document, we will fix $n,m$ non-negative integers. We will write $Q_0=[0,1]^{n+m}$ and points of $\RR^{n+m}$ as $(x,y)$ where $x\in \RR^n$ and $y\in \RR^m$. Later on, we will specialize to the case $n=2$ to prove Theorem \ref{thm:main}.

Our theorem on lower-bounding mapping content is an extension of a result of Kinneberg \cite[Corollary 4.2]{kinneberg2016discrete}. Kinneberg's result gives a simple way to lower bound the Hausdorff content $\HH^d_\infty(f([0,1]^d))$ of the image of a continuous map defined on the unit cube, and we extend his methods to apply to the mapping content $\HH^{n,m}_\infty$. 

Continuing to set $Q_0 = [0,1]^{n+m}$, we let 
\begin{equation}\label{eq:faces}
(F_1, F'_1), (F_2, F'_2), \dots, (F_{n+m}, F'_{n+m})
\end{equation}
 denote all the pairs of opposite faces of $\partial Q_0$. Thus,
$$ F_k = \{(t_1, t_2, \dots, t_{k-1}, 0, t_{k+1}, \dots, t_{n+m}) : t_i\in [0,1]\}$$
and
$$ F'_k = \{(t_1, t_2, \dots, t_{k-1}, 1, t_{k+1}, \dots, t_{n+m}) : t_i\in[0,1]\}.$$

\begin{maintheorem}\label{thm:lowerbound}
Let $f\colon Q_0 \rightarrow Y$ be a continuous map into a metric space. Then
\begin{equation}\label{eq:lowerbound}
 \HH^{n,m}_\infty(f,Q_0) \geq \prod_{k=1}^n \dist_{Y}(f(F_k), f(F'_k)).
\end{equation}
Moreover, the same lower bound holds for the quantity $\hat{\HH}^{n,m}_\infty$ defined in \eqref{eq:arbdef}:
\begin{equation}\label{eq:lowerbound2}
 \hat{\HH}^{n,m}_\infty(f,Q_0) \geq \prod_{k=1}^n \dist_{Y}(f(F_k), f(F'_k)).
\end{equation}
\end{maintheorem}
The quantity  $\hat{\HH}^{n,m}_\infty$ is a variant of $\HH^{n,m}_\infty$ introduced in \cite[Section 1.4.2]{David-Schul:harder-sarder} that is more convenient in some circumstances; see \eqref{eq:arbdef} below.

\begin{remark}
In the case $m=0$, we have $\HH^{n,m}_\infty(f,Q_0)=\HH^{n}_\infty(f(Q_0))$, so this agrees with \cite[Corollary 4.2]{kinneberg2016discrete}. Indeed, our proof follows the same path as Kinneberg's proof, including relying on his \cite[Proposition 4.1]{kinneberg2016discrete}.
\end{remark}
\begin{remark}
Of course, the fact that we use only the first $n$ pairs of sides in Theorem \ref{thm:lowerbound} is irrelevant, as one can always reduce to this case.
\end{remark}

Theorem \ref{thm:lowerbound} is of interest in its own right, providing a simple lower bound for mapping content which then implies the existence of Hard Sard sets as described above. In this paper, it will be used to prove a continuity statement for mapping content.

For this, we will use the following notion of distance for Lipschitz maps into arbitrary metric spaces, which is essentially an ad hoc notion of Gromov-Hausdorff convergence which suffices for our purposes.

\begin{definition}
Let $f\colon Q_0 \rightarrow Y$ and $g\colon Q_0 \rightarrow Z$ be continuous functions onto metric spaces $Y$ and $Z$. For $\epsilon>0$, we will say that $\dist(f,g)<\epsilon$ if there are isometric embeddings
$$ \iota_Y: Y \rightarrow \ell^\infty \text{ and } \iota_Z: Z \rightarrow \ell^\infty$$
such that
$$ \sup_{x\in Q_0} \| \iota_Y(f(x)) - \iota_Z(g(x))\|_\infty <\epsilon.$$
\end{definition}
\begin{remark}
To be concrete and avoid technicalities, we assume in the definition of $\dist$ that $Y$ and $Z$ are the full images of $f$ and $g$, respectively, hence the word ``onto'' in the definition. Of course, one can always reduce to this case.
\end{remark}

Our continuity statement for mapping content is then the following result.

\begin{maintheorem}\label{thm:contentzero}
Let $Q_0=[0,1]^{n+m}$. Let $f_i\colon Q_0 \rightarrow Y_i$ be a sequence of $1$-Lipschitz maps onto metric spaces $Y_i$, and $f\colon Q_0 \rightarrow Y$ another $1$-Lipschitz map onto a metric space $Y$.
Assume that $\dist(f_i, f) \rightarrow 0$ as $n\rightarrow \infty$. Then
$$ \HH^{n,m}_\infty(f,Q_0)=0 \text{ if and only if } \HH^{n,m}_\infty(f_i, Q_0)\rightarrow 0 \text{ as } i\rightarrow \infty.$$

\end{maintheorem}

Theorem \ref{thm:contentzero} will then combine with Theorem \ref{thm:EGH} to prove Theorem \ref{thm:main} by a compactness argument.

\subsection*{Acknowledgments}
The authors thank Behnam Esmayli and Piotr Haj\l asz for comments, and for sharing an early draft of \cite{EGH-personal}.

\section{Preliminaries}\label{sec:prelim}

\subsection{Dyadic cubes}

We write $Q_0$ for the unit cube in $\RR^{d}$, with $d$ generally understood from context, i.e.,
$$ Q_0 = [0,1]^{d}.$$
We write $\Delta$ for the collection of all dyadic cubes $Q\subseteq Q_0$, and $\Delta_k$ for the collection of those dyadic cubes with side length $2^{-k}$. 

If $Q\in\Delta$, we write $\side(Q)$ for the side-length of $Q$. Thus, $\side(Q)=2^{-k}$ if and only if $Q\in\Delta_k$.

If $Q\in\Delta$ and $C>0$, we write $CQ$ for a cube with the same center but $C$ times the side length. In particular, if $C$ is an odd positive integer, then $CQ$ is a union of $C^d$ distinct cubes of the same side length as $Q$.

Lastly, we call a collection of cubes ``almost-disjoint'' if they have disjoint interiors. Such collections arise in the definition of $\HH^{n,m}_\infty$.

\subsection{Metric derivatives}\label{subsec:metricderivative}

Let $X$ be a metric space and $f\colon \RR^d \rightarrow X$ a $1$-Lipschitz function. We will use  results from \cite{David-Schul:harder-sarder}, which rely on results and notation from \cite{AS14}, which were in turn inspired by the idea of metric differentiability in \cite{Kirchheim}.

For a cube $Q\subseteq \RR^d$ let
$$\md_f(Q) :=\frac{1}{\side(Q)} \inf _{\| \cdot\|} \sup_{x, y \in Q} \left| d(f(x),f(y)) - \|x-y\|\right|,$$
where the infimum is taken over all seminorms $\|\cdot\|$ on $\RR^d$. If the function $f$ is understood, we will simply write $\md(Q)$.
The quantity $\md_f(Q)$ measures how well the pullback of the distance in $X$ under $f$ can be approximated by a seminorm in $Q$. For metric space valued functions, it serves as a replacement for measuring ``deviation from linearity''. This quantity is appealing because of the result of \cite{AS14}, which is a quantitative differentiation result for Lipschitz mappings into metric spaces:
\begin{theorem}[\cite{AS14}, Theorem 1.1]\label{thm:quantdiff}
Let $X$ be a metric space and $f\colon \RR^d \rightarrow X$ a $1$-Lipschitz function. Let $\epsilon>0$ and $C_0>0$. Then
$$ \sum \{ |Q| : Q\in\Delta, \md_f(C_0 Q) > \epsilon\} \leq C_{\epsilon,d}.$$
The constant $C_{\epsilon,d}$ depends only on $\epsilon$, $C_0$, and $d$ but not on the space $X$ or the function $f$.
\end{theorem}
Here $|Q|$ refers to the $d$-dimensional Lebesgue measure of $Q$. While we do not use Theorem \ref{thm:quantdiff} directly below, it is behind the proof of Lemma \ref{lem:goodcube} from \cite{David-Schul:harder-sarder}.

For a cube $Q$, we let $\|\cdot\|_Q$ denote a seminorm that gives rise to the infimum in the definition above of $\md$ (repressing $f$ in the notation).
The existence of a minimizer follows from the observation that the norm is determined by its unit ball, and such sets are compact under the Hausdorff metric.

\subsection{Content, measure, and variations}\label{subsec:content}
We use $|E|$ to denote the Lebesgue measure of a subset $E$ of some Euclidean space, with the dimension understood from context.

For a subset $A$ of a metric space $X$, the $k$-dimensional Hausdorff content $\HH^k_\infty(A)$ is a well-used notion (see, e.g., \cite[Chapter 8]{He}) and is defined as follows:
$$ \HH^k_\infty(A) = \inf\sum_{U\in\mathcal{U}} \diam(U)^k,$$
where the infimum is taken over all open covers $\mathcal{U}$ of $A$. (Small modifications of the definition, e.g. to use covers by balls or arbitrary sets, yield comparable quantities.) The quantity $\HH^k_\infty$ is countably sub-additive, but not in general a measure.

The notion of mapping content $\HH^{n,m}_\infty$ was defined above in Definition \ref{def:mappingcontent}. As discussed in \cite{AS12, David-Schul:harder-sarder}, $\HH^{n,m}_\infty$ serves in some sense as a ``coarse'' substitute for the $L^1$-norm of the $n$-dimensional Jacobian of $f$. 

In \cite[Section 1.4.2]{David-Schul:harder-sarder}, a variation of mapping content was defined that uses arbitrary sets in the cover, rather than dyadic cubes.

\begin{equation}\label{eq:arbdef}
\hat{\HH}^{n,m}_\infty(f,A) = \inf \sum_i \HH^{n}_\infty(f(S_i))\diam(S_i)^m,
\end{equation}
where the infimum is over all countable covers $\{S_i\}$ of $A$ by \textit{arbitrary} subsets of $Q_0$. It is immediately clear that
\begin{equation}\label{eq:arbcontent}
\hat{\HH}^{n,m}_\infty(f,A) \lesssim_{n,m} \HH^{n,m}_\infty(f,A),
\end{equation}

As discussed in \cite{David-Schul:harder-sarder}, we do not know if the reverse inequality holds in \eqref{eq:arbcontent}. We do however know that if $\hat{\HH}^{n,m}_\infty(f,A)=0$ then
$\HH^{n,m}_\infty(f,A)=0$.  Indeed we have the following more quantitative statement.

\begin{theorem}[Corollary E of \cite{David-Schul:harder-sarder}]\label{cor:arbcontent}
For each $\delta>0$, there is a $\delta'>0$ with the following property:

If $f:Q_0\rightarrow X$ is a $1$-Lipschitz mapping into a metric space, and $A\subseteq Q_0$ has
$$ \HH^{n,m}_\infty(f,A) \geq \delta,$$
then
$$ \hat{\HH}^{n,m}_\infty(f,A) \geq \delta'.$$
The number $\delta'$ depends only on $\delta$, $n$, and $m$.
\end{theorem}

\section{Facts about seminorms}
Here we collect some lemmas about seminorms that we need, the main goal being Lemma \ref{lem:svd} below.

\begin{lemma}\label{lem:seminorm}
Let $\|\cdot\|$ be a seminorm on $\RR^{n+m}$. Then there is a linear map $A\colon \RR^{n+m} \rightarrow \RR^{n+m}$ such that $|Av| \approx \|v\|$ for all $v\in \RR^{n+m}$. The implied constant depends only on $n$, $m$.
\end{lemma}%
\begin{proof}
Let 
$$ V = (ker(\|\cdot\|))^\bot = \left(\{ v : \|v\|=0\}\right)^\bot,$$
a subspace of $\RR^{n+m}$ with some dimension $k\in\{0,\dots, n+m\}$. 

Let $A_0\colon \RR^{n+m} \rightarrow \RR^{n+m}$ be the orthogonal projection onto $V$. 
Notice that $\|A_0w\|=\|w\|$ for all $w\in \RR^{n+m}$, and that $\|\cdot\|$ is a a norm when restricted to $V$.

Let $K\subset V$ be the unit ball of $\|\cdot\|$.  A theorem of F. John says (see e.g. \cite{ball1992ellipsoids})
that there is an ellipsoid $E\subset K$ and such that 
$K\subset kE$, where $k$ is the dimension of $V$.
Let $T$ be the linear map taking $E$ to the unit ball of $\RR^k$, viewed as embedded in $\RR^{n+m}$.
Then for any unit vector $v\in V$ we have $1\leq |T(v)|\leq k$, and thus,
setting $A=T\circ A_0$ we have  for all $w\in \RR^{n+m}$ that  $\|w\|=\|A_0w\|\leq |Aw|\leq k\|A_0w\|=k\|w\|$.
Since $k\leq n+m$ this completes the proof. 
\end{proof}

The following fact about seminorms is now a consequence of Lemma \ref{lem:seminorm} and the singular value decomposition of matrices.

\begin{lemma}\label{lem:svd}
Let $\|\cdot\|$ be a seminorm on $\RR^{n+m}$. Suppose that there is an $n$-dimensional subspace $P\subseteq \RR^{n+m}$ and a constant $\delta>0$ such that 
$$ \|v\| \geq \delta |v| \text{ for all } v\in P.$$

Then there is an orthonormal basis $\{v_1, \dots, v_{n+m}\}$ of $\RR^{n+m}$ such that
$$ \left\|\sum_{i=1}^{n+m} a_i v_i\right\| \gtrsim \delta \sum_{i=1}^n |a_i|$$
for all choices of $a_1, \dots, a_{n+m}$ in $\RR$. The implied constant depends only on $n$, $m$.
\end{lemma}
\begin{proof}

Let $A\colon \RR^{n+m}\rightarrow \RR^{n+m}$ be as in Lemma \ref{lem:seminorm}. 

Let $A=USV^*$ denote the singular value decomposition of $A$, with singular values $\sigma_1\geq \dots \geq \sigma_{n+m}$. Thus, $U$ and $V^*$ are orthogonal $(n+m)\times (n+m)$ matrices and $S$ is a diagonal $(n+m)\times (n+m)$ matrix with the singular values along the diagonal. The max-min characterization of singular values says that the $n$th largest singular value, $\sigma_n$, can be computed by
$$ \sigma_n = \max_{\dim(T)=n} \min_{0\neq x\in T} \frac{|Ax|}{|x|} \gtrsim \max_{\dim(T)=n} \min_{0\neq x\in T} \frac{\|x\|}{|x|},$$
where the maximum is over all subspaces $T$ of dimension $n$ in $\RR^{n+m}$. 

Taking $T=P$, we obtain
$$ \sigma_n \gtrsim \delta>0,$$
with implied constant depending only on $n,m$.

For $i=1, \dots, n+m$, let $v_i =(V^*)^{-1}(e_i)$, where $\{e_i\}$ is the standard basis of $\RR^{n+m}$. Then for any choice of $a_1, \dots, a_{n+m}$ in $\RR$, we have
\begin{align*}
\left\|\sum_{i=1}^{n+m} a_i v_i\right\| &\gtrsim \left|A\left(\sum_{i=1}^{n+m} a_i v_i\right)\right|\\
&= \left|\sum_{i=1}^{n+m} a_i Se_i\right|\\
&\gtrsim \delta \sum_{i=1}^n |a_i|,
\end{align*}
with implied constants depending only on $n$ and $m$.
\end{proof}

\section{Lower bounds for mapping content}
This section is devoted to the proof of Theorem \ref{thm:lowerbound}.

\subsection{A result of Kinneberg}
We describe here the set-up of Kinneberg \cite[Section 4]{kinneberg2016discrete} which we require.

Let $g\colon [0,1]^d \rightarrow X$ be a continuous map into a metric space. Let $\{U_i\}_{i\in I}$ be an open cover of $g([0,1]^d)$. A ``chain'' of open sets from this collection is a finite sequence $U_{i_1}, U_{i_2}, \dots, U_{i_T}$ such that consecutive sets intersect. We say that the chain \emph{connects} two subsets  $E,F\subseteq g([0,1]^d)$ if $E$ intersects $U_{i_1}$ and $F$ intersects $U_{i_T}$.

For each $1\leq k \leq d$, let 
$$ w_k\colon I \rightarrow [0,\infty)$$
be a ``weight function'' defined on the sets of the cover. 

For subsets $E,F\subseteq g([0,1]^d)$, their ``combinatorial distance'' with respect to the above choices is defined as
$$ \dist_{w_k}(E, F) = \inf \sum_{t=1}^T w_k(i_t), $$
where the infimum is taken over all chains $U_{i_1}, \dots, U_{i_T}$ connecting $E$ to $F$.

Proposition 4.1 of \cite{kinneberg2016discrete} is then the following:
\begin{proposition}[Kinneberg \cite{kinneberg2016discrete}]\label{prop:kinneberg}
For any continuous $g\colon [0,1]^d\rightarrow X$ and any open cover $\{U_i\}_I$ and weight functions $w_k$ as defined above, we have
\begin{equation}\label{eq:kinneberg}
 \sum_{i\in I} \left( \prod_{k=1}^d w_k(i) \right) \geq \prod_{k=1}^d \dist_{w_k}(g(F_k), g(F'_k)).
\end{equation}
\end{proposition}
Recall that $\{F_k, F'_k\}$ denote opposite faces of the boundary of the unit cube, as defined in \eqref{eq:faces}.

\subsection{The proof of Theorem \ref{thm:lowerbound}}
\begin{proof}[Proof of Theorem \ref{thm:lowerbound}]
We will prove the second statement \eqref{eq:lowerbound2} of Theorem \ref{thm:lowerbound}, and then indicate the minor modifications needed to obtain the first statement \eqref{eq:lowerbound}.

Let $f$ be as in the statement of the theorem. 
Let $\{S_i\}_{i\in I}$ be an arbitrary cover of $Q_0$ by sets $S_i$. For each $i\in I$, let $\{U^i_j\}_{j\in J_i}$ be an arbitrary open cover of $f(S_i)$. Fix $\eta>0$ arbitrary.

For each $i\in I$ and $j\in J_i$, let
$$ V^i_j =U^i_j \times N_{\eta}(\text{Proj}_y(S_i)) \subseteq Y \times [0,1]^m.$$
Here $\text{Proj}_y$ denotes the map $(x,y)\mapsto y$ from $[0,1]^{n+m}$ to $[0,1]^m$, and $N_\eta(\cdot)$ denotes the open $\eta$-neighborhood of a set in $[0,1]^m$. The purpose of taking the neighborhood $N_\eta$ in $[0,1]^m$ is just to ensure that each set $V^i_j$ is open.

We now define weights $w_k(i,j)\in [0,\infty)$ for $1\leq k \leq n+m$, $i\in I$, and $j\in J_i$. 
\[ w_k(i,j) = \begin{cases} 
      \diam(U^i_j) & 1\leq k \leq n \\
     \diam(S_i)+2\eta & n+1\leq k\leq n+m 
   \end{cases}
\]
We will apply Theorem \ref{eq:kinneberg} to the function $h\colon Q_0 \rightarrow Y \times [0,1]^m$ defined by $ h(x,y) = (f(x,y),y).$
We equip $Y\times [0,1]^m$ with the metric $d_{Y\times [0,1]^m}$ defined as $d_{Y\times [0,1]^m}((p,t),(q,s)) = \max\{d_Y(p,q), |t-s|\}$. In this application, the open cover of Theorem \ref{eq:kinneberg} will be $\{V^i_j\}$, and the weights $w_k(i,j)$ will be as defined above. The left-hand side of \eqref{eq:kinneberg} is then
\begin{equation}\label{eq:inf} 
\sum_{i\in I} \sum_{j\in J_I} \prod_{k=1}^{n+m} w_k(i,j) = \sum_{i\in I} \left(\sum_{j\in J_i} \diam(U^i_j)^n\right) (\diam(S_i)+2\eta)^m.
\end{equation}
After taking the infimum over all choices of $\{S_i\}$, $\{U^i_j\}$, and $\eta$, this yields exactly $\hat{\HH}^{n,m}_\infty(f,Q_0)$.

We now lower bound the right-hand side of \eqref{eq:kinneberg}. First, for $1 \leq k \leq n$, we have
\begin{align*}
\dist_{w_k}(h(F_k), h(F'_k)) &= \inf\{ \sum_t w_k(V^{i_t}_{j_t}) : \{V^{i_t}_{j_t}\} \text{ connects } h(F_k), h(F'_k)\}\\
&=\inf\{ \sum_t \diam(U^i_j) : \{V^{i_t}_{j_t}\} \text{ connects } h(F_k), h(F'_k)\}\\
&\geq \inf\{ \sum_t \diam(U^i_j) : \{U^{i_t}_{j_t}\} \text{ connects } f(F_k), f(F'_k)\}\\
&\geq \dist_Y(f(F_k), f(F'_k)).
\end{align*}

Next, consider $n+1 \leq k \leq n+m$. Let 
$$ A = \{ (y_1, \dots, y_{k-n-1}, 0, y_{k-n+1}, \dots, y_m\} \subseteq [0,1]^m$$
and
$$ B = \{ (y_1, \dots, y_{k-n-1}, 1, y_{k-n+1}, \dots, y_m\} \subseteq [0,1]^m.$$
We then have
\begin{align*}
\dist_{w_k}(h(F_k), h(F'_k)) &= \inf\{ \sum_t w_k(V^{i_t}_{j_t}) : \{V^{i_t}_{j_t}\} \text{ connects } h(F_k), h(F'_k)\}\\
&=\inf\{ \sum_t (\diam(S_i)+2\eta) : \{V^{i_t}_{j_t}\} \text{ connects } h(F_k), h(F'_k)\}\\
&\geq \inf\{ \sum_t (\diam(S_i)+2\eta) : \{ N_{\eta}(\text{Proj}_{k}(S_i))\} \text{ connects } A, B\}\\
&\geq 1
\end{align*}
It follows that the right-hand side of \eqref{eq:kinneberg} is bounded below by
$$ \prod_{k=1}^n \dist_Y(f(F_k), f(F'_k)),$$
independently of the choice of sets $\{S_i\}$, open sets $\{U^i_j\}$, and $\eta>0$. Taking the infimum over all these choices, we obtain \eqref{eq:lowerbound2}.

The proof of \eqref{eq:lowerbound} is the same, with the following minor modifications: The sets $S_i$ are replaced by almost-disjoint dyadic cubes $Q_i$. The role of the open sets $U^j_i$ and parameter $\eta>0$ remain the same. The weight $w_k(i,j)$ is the same as before if $1\leq k \leq n$ and changes to $\side(Q_i)+2\eta$ if $n+1\leq k \leq n+m$. 

With these modifications, taking the infimum over all collections $\{Q_i\}$, $\{U^i_j\}$, and parameters $\eta>0$ in the analog of \eqref{eq:inf} yields $\HH^{n,m}_\infty(f,Q_0)$, and the calculations with the modified weights $w_k(i,j)$ yield the same lower bound. 
\end{proof}

Theorem \ref{thm:lowerbound} will enter the proof of Theorem \ref{thm:contentzero} through the following corollary.

\begin{corollary}\label{cor:positivecontent}
Let $f\colon Q_0\rightarrow Y$ be $1$-Lipschitz and $c,\eta>0$. Let $Q\subseteq Q_0$ be a cube, 
and let $\|\cdot\|$ be a seminorm such that
$$ \left| d(f(x),f(y)) - \|x-y\| \right| < \eta\side(3Q) \text{ for all } x,y\in 3Q.$$

Suppose that there is an $n$-plane $P\subseteq \RR^{n+m}$ such that $\|v\| \geq c |v|$ for all $v\in P$.

If $\eta$ is sufficiently small, depending only on $n,m,c$, then 
$$ \HH^{n,m}_\infty(f,Q) \gtrsim |Q|,$$
with implied constant depending only on $n$, $m$, and $c$.
\end{corollary}
\begin{proof}
Let $v_1, \dots, v_{n+m}$ be an orthonormal basis of $\mathbb{R}^{n+m}$ ``adapted'' to $\|\cdot\|$, as provided by Lemma \ref{lem:svd}. Let $Q'\subseteq Q$ be a rotated cube with $\side(Q')\approx_{n,m} \side(Q)$ oriented along the axes $v_1, \dots, v_{n+m}$. Let $\{(F_k, F'_k)\}_{k=1}^{n+m}$ denote the pairs of opposite faces of $Q'$.

Let $1\leq k \leq n$ and let $x\in F_k, x'\in F'_k$. Let $v=x'-x$. Writing $v=\sum_{i=1}^{n+m} a_i v_i$ in the new basis, we have $a_k\gtrsim \side(Q)$. By choice of seminorm, 
$$ d(f(x), f(x')) \geq \|v\| - 3\eta\side(Q).$$
By Lemma \ref{lem:svd}, $\|v\| \gtrsim c|a_k|\gtrsim c\side(Q)$, with implied constant depending only on $n,m$. Thus, if $\eta$ is chosen sufficiently small depending on $n$, $m$, and $c$, then $d(f(x),f(x'))$ is bounded below away from zero by $a\side(Q)$, for some constant $a$ depending only on $n$, $m$, and $c$.

This proves that 
$$ \dist(f(F_k), f(F'_k)) \geq a\side(Q) \text{ for } 1\leq k \leq n,$$
and therefore by statement \eqref{eq:lowerbound2} of Theorem \ref{thm:lowerbound} (rescaled and rotated to apply to $Q'$) that
$$ \hat{\HH}^{n,m}_\infty(f,Q')\gtrsim |Q'|,$$
with implied constant depending on $n$, $m$, and $c$.

Therefore, we have
$$ \HH^{n,m}_\infty(f,Q_0) \gtrsim_{n,m} \hat{\HH}^{n,m}_\infty(f,Q_0) \geq \hat{\HH}^{n,m}_\infty(f,Q')\gtrsim_{n,m,c} |Q'| \gtrsim_{n,m} |Q|.$$
\end{proof}

\section{Zero mapping content passes to the limit}
This section is devoted to the proof of Theorem \ref{thm:contentzero}. We begin with some lemmas. The first is Lemma 7.3 of \cite{David-Schul:harder-sarder} (stated in a slightly different way).

\begin{lemma}\label{lem:goodcube}
Given $\delta>0$, there is a constant $c>0$, depending only on $n$, $m$, and $\delta$, with the following property.

Let $Q_0=[0,1]^{n+m}$ and let $f\colon Q_0 \rightarrow Y$ be $1$-Lipschitz. Suppose that $\HH^{n,m}_\infty(f,Q_0)=\delta>0$ and $\eta>0$. 
Then there is a cube $Q\subseteq Q_0$ such that
$$ \md_f(3Q)<\eta$$
and an $n$-plane $P$ in $\RR^{n+m}$ such that
\begin{equation}\label{eq:plane}
\|v\|_Q \geq c|v|\text{ for all } v\in P.
\end{equation}
Moreover, the size of the cube $Q$ can be bounded below by a constant depending only on $n$, $m$, $\delta$, and $\eta$.
\end{lemma}

\begin{remark}
In fact, the proof in \cite{David-Schul:harder-sarder} shows that $P$ can be chosen to be a coordinate plane, though we do not need this here. 
\end{remark}





\begin{proof}[Proof of Theorem \ref{thm:contentzero}]
Let $\{f_i\colon Q_0 \rightarrow Y_i\}$ and $f\colon Q_0 \rightarrow Y$ be as in the statement of Theorem \ref{thm:contentzero}.

Suppose first that $\HH^{n,m}_\infty(f,Q_0)=\delta>0$; we will show that $\HH^{n,m}_\infty(f_i,Q_0)\rightarrow 0$.

Let $c>0$, depending on $\delta,n,m$, be as in Lemma \ref{lem:goodcube}. Let $\eta$ be chosen sufficiently small, depending on $n,m,c$, as required by Corollary \ref{cor:positivecontent}.

By Lemma \ref{lem:goodcube}, we may find a cube $Q\subseteq Q_0$ such that 
$$ \md_f(3Q) < \eta/10$$
and such that there is an $n$-plane $P$ as in \eqref{eq:plane}, with respect to a seminorm $\|\cdot\|_Q$ on $\RR^{n+m}$ such that
$$ |d(f(x),f(y)) - \|x-y\|_Q| < \frac{\eta}{10}\side(3Q) \text{ for all } x,y\in 3Q.$$

Choose $i_0$ sufficiently large so that $\dist(f_i,f)<\frac{\eta}{10}\side(Q)$ for all $i\geq i_0$. 
The same seminorm $\|\cdot\|_Q$ satisfies
$$ |d(f_i(x) , f_i(y)) - \|x-y\|_Q| < \frac{3\eta}{10}\side(3Q) \text{ for all } x,y\in 3Q \text{ and } i\geq i_0$$
It follows from Corollary \ref{cor:positivecontent} that
$$ \HH^{n,m}_\infty(f_i, Q) \gtrsim_{n,m,\delta} |Q|$$
for all $i$ sufficiently large, which contradicts the assumption that $\HH^{n,m}_\infty(f_i,Q_0)\rightarrow 0$.

For the converse statement, we continue to assume that $\dist(f_i,f)\rightarrow 0$, but now we suppose that $\HH^{n,m}_\infty(f,Q_0)=0$ and that $\HH^{n,m}_\infty(f_i,Q_0)$ does not tend to zero. In that case, passing to a subsequence (which we rename $\{f_i\}$ to keep the notation the same) yields $\delta>0$ such that $\HH^{n,m}_\infty(f_i,Q_0)\geq\delta>0$. Exactly as before, we set $c=c(\delta,n,m)>0$ as in Lemma \ref{lem:goodcube} and $\eta=\eta(n,m,c)$ sufficiently small as required by Corollary \ref{cor:positivecontent}.

It follows that for each $i$, there is a dyadic cube $Q_i\subseteq Q_0$ and an $n$-plane $P_i$ as in \eqref{eq:plane}, with respect to a seminorm $\|\cdot\|_i$ on $\RR^{n+m}$ such that
$$ |d(f_i(x),f_i(y)) - \|x-y\|_i| < \frac{\eta}{10}\side(3Q_i) \text{ for all } x,y\in 3Q_i.$$
Moreover, Lemma \ref{lem:goodcube} guarantees that $\side(Q)\geq s_0$, for some $s_0$ depending only on $n,m,\delta,\eta$, and hence $n,m,\delta$.

Choose any $i$ sufficiently large so that $\dist(f,f_i)<\eta s_0 /10 \leq \eta\side(Q_i)/10$. The same seminorm $\|\cdot\|_i$ satisfies
$$ |d(f(x) , f(y)) - \|x-y\|_i| < \frac{3\eta}{10}\side(3Q) \text{ for all } x,y\in 3Q_i \text{ and } i\geq i_0$$
It therefore again follows from Corollary \ref{cor:positivecontent} that
$$ \HH^{n,m}_\infty(f,Q_0) \geq \HH^{n,m}_\infty(f, Q_i) \gtrsim_{n,m,\delta} |Q_i| > 0$$
contradicting the assumption that $\HH^{n,m}_\infty(f,Q_0)=0$.

\end{proof}

\section{Proofs of Theorem \ref{thm:alldim} and Theorem \ref{thm:main}}
In this section, we prove the two approximation theorems, Theorems \ref{thm:alldim} and \ref{thm:main}. The latter will be an immediate corollary of the former and Theorem \ref{thm:EGH}.

We first need a Gromov-Hausdorff type compactness lemma.
\begin{lemma}\label{lem:GHcompact}
Fix $d\in \mathbb{N}$. Let $f_i\colon [0,1]^d \rightarrow Y_i$ be a sequence of $1$-Lipschitz mappings onto metric spaces $Y_i$. Then there is a $1$-Lipschitz mapping $f\colon [0,1]^d \rightarrow Y$ onto a metric space and a subsequence $\{f_{i_k}\}$ such that
$$ \dist(f_{i_k},f) \rightarrow 0.$$
\end{lemma}
\begin{proof}

We will first use Gromov's compactness theorem (see \cite[Theorem 7.4.15]{burago2001course}) to argue that a subsequence of the spaces $Y_i$ converge in the Gromov-Hausdorff topology. First of all, each $Y_i$ has 
\begin{equation}\label{eq:gh1}
\diam(Y_i)\leq \diam([0,1]^d)=\sqrt{d} \text{ uniformly in } i,
\end{equation}
as a $1$-Lipschitz image of the unit cube. Second of all, for every $\epsilon>0$, let $N(X,\epsilon)$ be the maximal size of an $\epsilon$-separated set in a metric space $X$. Observing that the pre-image of an $\epsilon$-separated set under a $1$-Lipschitz map is again $\epsilon$-separated, we have
\begin{equation}\label{eq:gh2}
N(Y_i,\epsilon) \leq N([0,1]^d,\epsilon) \lesssim \epsilon^{-d} \text{ uniformly in } i.
\end{equation}
The properties \eqref{eq:gh1} and \eqref{eq:gh2} of the sequence $\{Y_i\}$ imply that the sequence contains a Gromov-Hausdorff convergent subsequence, by \cite[Theorem 7.4.15]{burago2001course}. For convenience, let us continue to label this subsequence $\{Y_i\}$.

We may now find isometric embeddings $\iota_i\colon Y_i \rightarrow \ell^\infty$ and a compact subset $Y\subseteq \ell^\infty$ such that 
$$ d_{\text{Haus}}(\iota_i(Y_i), Y) \rightarrow 0,$$
where $d_{\text{Haus}}$ denotes the Hausdorff distance in $\ell^\infty$. (See \cite[Proposition 2.8]{heinonen2003geometric}.)

By a standard diagonalization argument, we may now extract a subsequence $\{f_{i_k}\}$ of our original sequence and a $1$-Lipschitz map $f\colon Q_0 \rightarrow Y$ such that
$$  \sup_{Q_0}\|\iota_{i_k}\circ f_{i_k} - f\|_\infty \rightarrow 0.$$
It follows that
$$ \dist(f_{i_k}, f) \rightarrow 0.$$
\end{proof}

We will also use the following simple fact.
\begin{lemma}\label{lem:monotone}
Let $f\colon Q_0 \rightarrow Y$ and $\phi\colon Y \rightarrow Z$ be $1$-Lipschitz maps. Then
$$ \HH^{n,m}_\infty(\phi\circ f,Q_0) \leq \HH^{n,m}_\infty(f,Q_0).$$
\end{lemma}
\begin{proof}
This follows immediately from the definition of mapping content and the fact that $$\HH^n_\infty(\phi(f(Q))) \leq \HH^n_\infty(f(Q))$$ for every cube $Q$.
\end{proof}

\begin{proof}[Proof Theorem \ref{thm:alldim}]
We begin with the first statement in the theorem: small mapping content implies close to factoring through a tree. Suppose that the first statement in the theorem were false. There would then be an $\epsilon>0$ and a sequence $f_i\colon Q_0 \rightarrow \ell^\infty$ of $1$-Lipschitz maps such that
$$ \HH^{n,m}_\infty(f_i, Q_0) \rightarrow 0 $$
but
\begin{equation}\label{eq:notclose}
\sup_{Q_0} \|f_i - g\|_\infty \geq \epsilon
\end{equation}
for all $i$ and all $1$-Lipschitz maps $g\colon Q_0 \rightarrow \ell^\infty$ with $\HH^{n,m}_\infty(g,Q_0)=0$. 

By Lemma \ref{lem:GHcompact}, there is a $1$-Lipschitz map $f$ from $Q_0$ onto a metric space $Y$ such that
$$ \dist(f_i, f)\rightarrow 0.$$
By Theorem \ref{thm:contentzero}, we have $\HH^{n,m}_\infty(f,Q_0)=0$.

Choose $i$ large such that $\dist(f_i, f)<\epsilon$, and write $Y_i=f_i(Q_0)$. There are isometric embeddings
$$ \iota\colon Y_i \rightarrow \ell^\infty$$
and
$$ \iota'\colon Y \rightarrow \ell^\infty$$
such that 
$$ \sup_{Q_0} \|\iota \circ f_i - \iota' \circ f\|_\infty < \epsilon.$$

Let $j\colon \iota(Y_i) \rightarrow Y_i \subseteq \ell^\infty$ be the inverse of $\iota$. We may extend $j$ to a $1$-Lipschitz map from $\ell^\infty$ to itself. Let $g = j \circ \iota' \circ f$. Note that, like $f$, $g$ is a $1$-Lipschitz map with $\HH^{n,m}_\infty(g,Q_0)=0$, by Lemma \ref{lem:monotone}. 

In addition, 
\begin{align*}
\sup_{Q_0}\|f_i - g\|_\infty &= \sup_{Q_0} \| j \circ \iota \circ f_i - j \circ \iota' \circ f\|_\infty\\
&\leq \sup_{Q_0} \| \iota \circ f_i - \iota' \circ f\|_\infty\\
&<\epsilon
\end{align*}
This contradicts \eqref{eq:notclose} and so completes the proof of the first half of Theorem \ref{thm:alldim}.

We now consider the second, converse statement. Suppose that the second statement in the theorem were false. There would then be an $\epsilon>0$ and two sequences $f_i, g_i \colon Q_0 \rightarrow \ell^\infty$ of $1$-Lipschitz maps such that
\begin{itemize}
\item $\sup_{Q_0} \|f_i - g_i\|_\infty \rightarrow 0$,
\item $\HH^{n,m}(g_i,Q_0)=0$ for each $i$, and
\item $\HH^{n,m}_\infty(f_i,Q_0)\geq \epsilon$ for each $i$ and some fixed $\epsilon>0$.
\end{itemize}
By passing to a subsequence, we may again assume that there is a $1$-Lipschitz map $f\colon Q_0 \rightarrow Y$ such that $\dist(f_i,f)\rightarrow 0$. 

There are isometric embeddings $\iota_i\colon f_i(Q_0) \rightarrow \ell^\infty$, $\iota\colon f(Q_0)\rightarrow\ell^\infty$ such that
$$ \sup_{Q_0} \| \iota_i\circ f_i - \iota\circ f\|.$$
As before, we may extend $\iota_i$ and $\iota$ to be $1$-Lipschitz maps defined on all of $\ell_\infty$. Let $j$ be the inverse of $\iota|_{f(Q_0)}$, extended to be a $1$-Lipschitz map on $\ell_\infty$.

It follows that
\begin{align*}
\sup_{Q_0}\|f-j\circ \iota_i \circ g_i\|_\infty &= \sup_{Q_0}\|j\circ \iota \circ f - j\circ \iota_i \circ g_i\|_\infty\\
&\leq \sup_{Q_0}\|\iota \circ f - \iota_i \circ g_i\|_\infty\\
&\leq \sup_{Q_0}\|\iota \circ f - \iota_i \circ f_i\|_\infty + \sup_{Q_0}\|\iota_i \circ f_i - \iota_i \circ g_i\|_\infty\\
&\leq \sup_{Q_0}\|\iota \circ f - \iota_i \circ f_i\|_\infty + \sup_{Q_0}\| f_i - g_i\|_\infty\\
&\rightarrow 0
\end{align*}
Therefore, $\dist(f, j\circ\iota_i\circ g_i)\rightarrow 0$. Since $\HH^{n,m}_\infty(g_i, Q_0)=0$ for each $i$, the same holds for $\iota_i\circ g_i$ by Lemma \ref{lem:monotone}. It follows from Theorem \ref{thm:contentzero} that $\HH^{n,m}_\infty(f,Q_0)=0$. 

Again by Theorem \ref{thm:contentzero}, we must therefore have $\HH^{n,m}_\infty(f_i, Q_0)\rightarrow 0$, but this contradicts one of the defining properties of $f_i$.
\end{proof}

\begin{proof}[Proof of Theorem \ref{thm:main}]
The theorem now follows immediately from Theorem \ref{thm:alldim} and Theorem \ref{thm:EGH}.
\end{proof}

\bibliography{trees}{}
\bibliographystyle{plain}

\end{document}